\documentclass[final, 1p, times]{elsarticle}

\usepackage{amssymb}
\usepackage{amsthm}
\usepackage{amscd}
\usepackage{amsmath}
\usepackage{amsfonts}
\usepackage{graphicx}
\usepackage{color}
\usepackage{mathrsfs}
\usepackage{titletoc}
\usepackage{titlesec}
\newtheorem{theorem}{Theorem}
\newtheorem{lemma}{Lemma}
\newcommand{\f}{\left}
\newcommand{\ri}{\right}

\begin{document}		
	\begin{frontmatter}
	\title{Extremal functions for a class of trace Trudinger-Moser inequalities}
	\author{Mengjie Zhang}
	\ead{zhangmengjie@ruc.edu.cn}
	\address{School of Mathematics, Renmin University of China, Beijing 100872, P.R.China}		
	\begin{abstract}
In this paper, we concern trace Trudinger-Moser inequalities on a compact Riemann surface with smooth boundary.
This kind of inequalities were extensively studied by Osgood-Phillips-Sarnak \cite{OPS}, Liu \cite{Liu}, Li-Liu \cite{Li-Liu}, Yang \cite{Yang2006, Yang2007} and others.
We establish several trace Trudinger-Moser inequalities and obtain the corresponding extremals via the method of blow-up analysis.
The results in the current paper generalize those of Li-Liu \cite{Li-Liu} and Yang \cite{Yang2007, Yang-JDE-15}.

\end{abstract}
		
	\begin{keyword}
	Trudinger-Moser inequality, Riemann surface, blow-up analysis, extremal function.\\
    2010 MSC: 46E35; 58J05.
	\end{keyword}
		
	\end{frontmatter}

\section{Introduction and main results}

Let $\Omega\subseteq \mathbb{R}^2$ be a smooth bounded domain and $W_0^{1,2}(\Omega)$ be the completion of $C_0^{\infty}(\Omega)$ under the Sobolev norm
$\|\nabla_{\mathbb{R}^2} u\|_2^2= \int_{\Omega}{|\nabla_{\mathbb{R}^2} u|^2}dx,$
where $\nabla_{\mathbb{R}^2}$ is the gradient operator on ${\mathbb{R}^2}$ and $\|\cdot\|_2$ denotes the standard $L^2$-norm.
The classical Trudinger-Moser inequality \cite{Yudovich, Pohozaev, Peetre, Trudinger1967, Moser1970}, as the limit case of the Sobolev embedding, says
    \begin{equation}\label{Trudinger-Moser}
    \sup_{u\in W_0^{1, 2}(\Omega), \, \|\nabla_{\mathbb{R}^2} u\|_2\leq 1}
    \int_\Omega e^{\ \beta u^2}dx<+\infty, \ \forall \ \beta\leq 4\pi.
    \end{equation}
Moreover, $4\pi$ is called the best
constant for this inequality in the sense that when $\beta> 4\pi$, all integrals in (\ref{Trudinger-Moser}) are still finite, but the supremum is infinite.
It is interesting to know whether or not the supremum in (\ref{Trudinger-Moser}) can be attained.
For this topic, we refer the reader to Carleson-Chang \cite{C-C}, Flucher \cite{Flucher}, Lin \cite{Lin}, Adimurthi-Struwe \cite{A-Struwe}, Li \cite{Li-JPDE,Li-Sci}, Yang \cite{Yang-JFA-06, Yang-IJM}, 
Zhu \cite{ZhuJY}, Tintarev \cite{Tintarev} and the references therein.

Trudinger-Moser inequalities were introduced on Riemannian manifolds by Aubin \cite{A}, Cherrier \cite{C}, Fontana \cite{Fontana} and others.
In particular, let $( \Sigma , g )$ be a compact Riemann surface with smooth boundary $\partial \Sigma$ and $W^{1,2} ( \Sigma, g )$ be the completion of $C ^ { \infty } ( \Sigma)$ under the norm
$$\|u \| ^2_ { W^{1,2}( \Sigma, g) } =  \int _ { \Sigma }\left( u^2+| \nabla_g u | ^2\right)\, dv_g,$$
where $\nabla_g$ stands for the gradient operator on $(\Sigma, g)$.
Liu \cite{Liu} derived a trace Trudinger-Moser inequality in his doctoral thesis from the result of Osgood-Phillips-Sarnak \cite{OPS}: for all functions $u \in W^{1,2} ( \Sigma,g )$, there holds
\begin{equation}\label{1.1}
\log \int _ { \partial \Sigma } e ^ { u } d s _ { g } \leq \frac { 1 } { 4 \pi } \int _ { \Sigma} | \nabla_g u | ^ { 2 } d v _ { g } + \int _ { \partial \Sigma} u d s _ { g } + C
\end{equation}
for some constant $C$ depending only on $ (\Sigma,\ g)$.
Later Li-Liu \cite{Li-Liu} obtained a strong version of (\ref{1.1}), namely
 \begin{equation}\label{1}
\sup _ {u \in W^{1,2} ( \Sigma, g ),\ \int _ { \Sigma } | \nabla_g u | ^ { 2 } d v _ { g } = 1 , \int _ { \partial \Sigma } u d s _ { g } = 0 } \int _ { \partial \Sigma } \mathrm { e } ^ { \gamma u ^ { 2 } } d s _ { g } < + \infty
\end{equation}
for any $\gamma\leq \pi $.
This inequality is sharp in the sense that all integrals
in (\ref{1}) are still finite when $\gamma > \pi$, but the supremum is infinite.
Moreover, for any $\gamma\leq \pi $, the supremum is attained.

Another form of (\ref{1}) was established by Yang \cite{Yang2006}, say
\begin{equation*}\label{1.3}
\sup _ { u \in W^{1,2} ( \Sigma, g ),\ \int _ { \Sigma} | \nabla_g u | ^ { 2 }  d v _ g = 1 , \ \int_{\Sigma}u d v_g=0} \int _ { \partial \Sigma } e ^ { \pi u ^ { 2 } } d s _ { g } < + \infty.
\end{equation*}
Also, an improvement of (\ref{1}) was obtained by Yang \cite{Yang2007} as follows:
\begin{equation}\label{1.4}
\sup _ { u \in W^{1,2}(\Sigma, g),\  \int _ { \Sigma} | \nabla_g u | ^ { 2 } d v _ g = 1,\ \int_{\partial \Sigma} u d s_{g}=0 } \int _ { \partial \Sigma } \mathrm { e } ^ { \pi u ^ { 2 } \left( 1 +\alpha \|u\|^2_{L^2(\partial\Sigma)}  \right) }ds _ { g }<+\infty
\end{equation}
for all $0 \leq \alpha<\lambda_1(\partial \Sigma)$, where
\begin{equation}\label{la}
\lambda_1(\partial\Sigma)= \inf _ { u \in W^{1,2} ( \Sigma, g) , \int _ { \partial\Sigma } u ds_g = 0 , u \not\equiv 0 } \frac {\int _{\Sigma}  |\nabla_g u|^2 dv_g} {\int _{\partial\Sigma}  {u^2}\,ds_g  }
\end{equation}
is the first eigenvalue of the Laplace-Beltrami operator $\Delta _ { g}$ on the boundary $\partial \Sigma$.
 This inequality is sharp in the sense that all integrals
in (\ref{1.4}) are still finite when $ \alpha\geq\lambda_1(\partial \Sigma)$, but the supremum is infinite. Moreover, for sufficiently
small $\alpha>0$, the supremum in (\ref{1.4}) can be attained.

A different form of (\ref{1.4}) was also derived by Yang \cite{Yang-JDE-15}, namely
\begin{equation}\label{1.2}
\sup _{u \in W^{1,2}(\Sigma, g),\ \int_{\Sigma}(|\nabla_{g} u|^{2} -\alpha  u^{2} )d v_{g} \leq 1,\ \int_{\Sigma} u d v_{g}=0} \int_{\Sigma} e^{4 \pi u^{2}} d v_{g}<+\infty
\end{equation}
for all $0 \leq \alpha<\lambda_{1}(\Sigma)$, where
\begin{equation*}
\lambda_{1}(\Sigma)=\inf _{u \in W^{1,2}(\Sigma, g), \int _{\Sigma} u d v_{g}=0, u \not\equiv 0} \frac{\int_{\Sigma}|\nabla_{g} u|^{2} d v_{g}}{\int_{\Sigma} u^{2} d v_{g}}.
\end{equation*}
Further, he extended (\ref{1.2}) to the case of higher order eigenvalues.
Denote
\begin{equation*}
E^{\perp}=\left\{u \in W^{1,2}(\Sigma, g): \int_{\partial\Sigma} u v d s_{g}=0, \forall v \in E \right\},
\end{equation*}
where $E\subset W^{1,2}(\Sigma, g)$ is a function space.
For any positive integer $\ell$, we set
$$E_{\lambda_{l}(\Sigma)}=\left\{ u\in W^{1,2}(\Sigma, g): \Delta _ { g}=\lambda_l(\Sigma)u\right\},$$
\begin{equation*}
\lambda_{\ell+1}(\Sigma)=\inf _{u \in E_{\lambda _\ell(\Sigma)}^{\perp}, \int _{\Sigma} u d v_{g}=0, u \not\equiv 0} \frac{\int_{\Sigma}|\nabla_{g} u|^{2} d v_{g}}{\int_{\Sigma} u^{2} d v_{g}}
\end{equation*} 
and
$$E_{\ell}( \Sigma)=E_{\lambda_{1}(  \Sigma)} \oplus E_{\lambda_{2}( \Sigma)} \oplus \cdots \oplus E_{\lambda_{\ell}( \Sigma)}.$$
Then the supremum
\begin{equation}\label{1.5}
\sup _{u \in  E_{\ell}^{\perp}(\Sigma),\ \int_{\Sigma}(|\nabla_{g} u|^{2} -\alpha  u^{2}) d v_{g} \leq 1,\ \int_{\Sigma} u d v_{g}=0} \int_{\Sigma} e^{4 \pi u^{2}} d v_{g}<+\infty
\end{equation}
for all $0 \leq \alpha<\lambda_{\ell+1}(\Sigma)$; moreover the above supremum can be attained by some function $u_\alpha\in E_{\ell}^{\perp}(\Sigma)$.

In this paper, we consider trace Trudinger-Moser inequalities like (\ref{1.2}) and (\ref{1.5}).
Precisely we first have the following:
\begin{theorem}\label{T1}
Let $( \Sigma , g )$ be a compact Riemannian surface with smooth boundary $\partial\Sigma$ and $\lambda_1(\partial\Sigma)$
be defined by (\ref{la}).
For any $0\leq\alpha<\lambda_1(\partial\Sigma)$, we let
\begin{equation*}\label{H}
\mathcal { H } = \left\{ u \in W^{1,2} ( \Sigma, g) : \|u\|_{1, \alpha}\leq 1 \ \, \mathrm{and}\, \int _{\partial\Sigma}  {u}\,ds_g = 0 \right\},
\end{equation*}
where
\begin{equation}\label{1aa}
\|u\|_{1, \alpha}=\left(\int _{\Sigma}  |\nabla_g u|^2 dv_g -\alpha \int _{\partial\Sigma}  {u^2}ds_g \right)^{1/2}.
\end{equation}
Then the supremum
\begin{equation}\label{sup2}
\sup_{ u \in \mathcal { H } }\int _ { \partial\Sigma } e ^ {\pi u^ 2} ds_g
\end{equation}
is attained by some function $u_\alpha \in \mathcal{H}\cap C^{\infty} (\overline{ \Sigma})$.
     \end{theorem}

Moreover, we extend Theorem \ref{T1} to the case of higher order eigenvalues.
Let us introduce some notations.
For any positive integer $\ell$, we set
 $$E_{\lambda_{l}(\partial\Sigma)}=\left\{ u \in W^{1,2} ( \Sigma, g ) : \Delta _ { g } u= 0   \text { in }  (\Sigma,\ g) \text{ and }  \frac{ \partial u } { \partial {\mathbf{n}}}=\lambda _ { l} ( \partial \Sigma )u \text { on }  \partial \Sigma \right\},$$
 where $\mathbf{n}$ denotes the outward unit normal vector on $\partial \Sigma$,
\begin{equation}\label{lal}
\lambda_{\ell+1}(\partial\Sigma)= \inf _ { u \in E_{\lambda_\ell(\partial\Sigma)}^{\perp}, \int _ { \partial\Sigma } u ds_g = 0 , u \not\equiv 0 } \frac {\int _{\Sigma}  |\nabla_g u|^2 dv_g} {\int _{\partial\Sigma}  {u^2}\,ds_g  }
\end{equation}
and
\begin{equation}\label{l1}
  E_{\ell}( \partial\Sigma)=E_{\lambda_{1}(  \partial\Sigma)} \oplus E_{\lambda_{2}( \partial\Sigma)} \oplus \cdots \oplus E_{\lambda_{\ell}( \partial\Sigma)}.
\end{equation}
\begin{theorem}\label{2T1}
Let $( \Sigma , g )$ be a compact Riemannian surface with smooth boundary $\partial\Sigma$, $\ell$ be an positive integer and $\lambda_{\ell+1}(\partial\Sigma)$ be defined by (\ref{lal}).
For any $0\leq \alpha < \lambda _ { \ell + 1 } ( \partial \Sigma ) ,$
we let
\begin{equation}\label{S}
\mathcal { S } = \left\{ u \in E _ { \ell }^ { \perp } ( \partial \Sigma ) : \|u\|_{1, \alpha}\leq 1 \ \, \mathrm{and}\, \int _{\partial\Sigma}  {u}\,ds_g = 0 \right\},
\end{equation}
where $\|u\|_{1, \alpha}$ is defined as in (\ref{1aa}).
Then the supremum
\begin{equation*}
\sup_{ u \in \mathcal { S} }\int _ { \partial\Sigma } e ^ {\pi u^ 2} ds_g
\end{equation*}
is attained by some function $u_\alpha \in \mathcal{S}\cap C^{\infty} (\overline{ \Sigma})$.
     \end{theorem}

Clearly Theorem \ref{T1} generalizes that of (\ref{1}) and Theorem \ref{2T1} extends (\ref{1.5}) to the trace Trudinger-Moser inequality.
For theirs proofs,
we employ the method of blow-up analysis, which was originally used by Carleson-Chang\cite{C-C},
Ding-Jost-Li-Wang \cite{DJLW}, Adimurthi-Struwe \cite{A-Struwe}, Li \cite{Li-JPDE}, Liu \cite{Liu}, Li-Liu \cite{Li-Liu}, and Yang \cite{Yang2006, Yang2007}.
This method is now standard, for related works, we refer the reader to
Adimurthi-Druet \cite{AD},
do \'O-de Souza \cite{do-de2014,do-de2016}, Nguyen \cite{N2017,N2018}, Li \cite{Lixiaomeng}, Li-Yang \cite{L-Y}, Zhu \cite{Zhu}, Fang-Zhang \cite{F-Z} and Yang-Zhu \cite{Yang-Zhu2018,Yang-Zhu2019}.

In the remaining part of this paper, we prove Theorems \ref{T1} in Section 2 and Theorem \ref{2T1} in Section 3 respectively.

\section{The first eigenvalue case}
In this section, we will prove Theorem \ref{T1} in four steps:
firstly, we consider the existence of maximizers for subcritical functionals and give the corresponding Euler-Lagrange equation;
secondly, we deal with the asymptotic behavior of the maximizers through blow-up analysis;
thirdly, we deduce an upper bound of the supremum $\sup_{ u \in \mathcal { H } }\int _ { \partial\Sigma } e ^ {\pi u^ 2} ds_g$ under the assumption that blow-up occurs;
finally, we construct a sequence of functions to show Theorem \ref{T1} holds.
Here and in the sequel, we do not distinguish sequence and subsequence.

	\subsection{Existence of maximizers for subcritical functionals.}
Let $0 \leq \alpha < \lambda_1(\partial\Sigma)$.
Analogous to (\cite{Yang2006}, Lemma 4.1) and (\cite{Yang2007}, Lemma 3.2), we have the following:

\begin{lemma}\label{L1}
For any $0<\epsilon <\pi$,
the supremum
$\sup_{ u \in \mathcal { H } }\int _ { \partial\Sigma } e ^ {(\pi-\epsilon) u^ 2} ds_g$ is attained by some function $u_\epsilon \in \mathcal{H}$.
\end{lemma}

Moreover, it is not difficult to check that $u_{\epsilon}$ satisfies the Euler-Lagrange equation
\begin{equation}\label{e-u}
	\left\{
\begin{aligned}
&\Delta_g u _ { \epsilon } = 0 \,\,\, \mathrm{in}\,\,\, \Sigma ,\\
&{ \frac { \partial u _ { \epsilon } } { \partial \mathbf{n} } = \frac { 1 } { \lambda _ { \epsilon } } u _ { \epsilon } e ^ {(\pi- \epsilon) u _ { \epsilon } ^ { 2 } } + \alpha u _ { \epsilon } - \frac { \mu _ { \epsilon } } { \lambda _ { \epsilon } } \,\,\, \mathrm{ on } \,\,\, \partial \Sigma },\\
&\lambda _ { \epsilon }= {\int _ { \partial \Sigma } u _ { \epsilon } ^ { 2 } \mathrm { e } ^ { (\pi- { \epsilon }) u _ { \epsilon } ^ { 2 } } ds _ { g }} ,\\
&\mu _ { \epsilon }= \frac {1} { \ell ( \partial \Sigma ) } \int _ { \partial \Sigma } u _ { \epsilon } \mathrm { e } ^ { (\pi-\epsilon) u _ { \epsilon } ^ { 2 } }ds _ { g },
 \end{aligned} \right.
	\end{equation}
where $\Delta_g$ denotes the Laplace-Beltrami operator, $\mathbf{n}$ denotes the outward unit normal vector on $\partial\Sigma$ and $\ell ( \partial \Sigma )$ denotes the length of $\partial \Sigma$.
Applying elliptic estimates and maximum principle to (\ref{e-u}) respectively, we have $u _ { \epsilon }\in \mathcal{H}\cap C^{\infty} (\overline{ \Sigma})$ and $u_\epsilon \not\equiv 0$ on $\partial\Sigma$.
The fact of $e^t\leqslant 1+te^t$ for any $t\geqslant 0 $ implies that
$\lambda _ { \epsilon }
\geq 1/{(\pi-\epsilon)} {\int _ { \partial\Sigma }\left( e ^ { (\pi-{ \epsilon }) u _ { \epsilon } ^ { 2 } }-1 \right)\ ds_g},$
which gives
\begin{equation}\label{L2}
  \liminf_{\epsilon \rightarrow 0} \ \lambda _{\epsilon}>0.
\end{equation}
From (\ref{L2}), one gets ${\mu _{\epsilon}}/{\lambda _{\epsilon}}$ is bounded with respect to $\epsilon$.
In addition, we have
    \begin{equation}\label{3}
    {\lim_{\epsilon \rightarrow 0}}\int _ { \partial\Sigma } e ^ {(\pi-\epsilon) u_{\epsilon}^ 2} ds_g=\sup_{ u \in \mathcal { H } }\int _ { \partial\Sigma } e ^ {\pi u^ 2} ds_g
    \end{equation}
from Lebesgue's dominated convergence theorem.

\subsection{Blow-up analysis}
We now perform the blow-up analysis. Let $c_{\epsilon}=|u_{\epsilon}( x_\epsilon )| ={\max}_{\overline{\Sigma}}|u_{\epsilon}|$.
With no loss of generality, we assume in the following $c_{\epsilon}=u_{\epsilon}( x_\epsilon )\rightarrow + \infty$ and $x_\epsilon\rightarrow p$ as $\epsilon\rightarrow0$.
Applying maximum principle to (\ref{e-u}), we have $p \in \partial \Sigma$.

\begin{lemma}\label{L3}
There hold $u_{\epsilon}\rightharpoonup 0$ weakly in $W^{1,2}( \Sigma , g)$ and $u_{\epsilon}\rightarrow 0$ strongly in $ L^2( \partial\Sigma, g)$ as $\epsilon\rightarrow0$.
Furthermore, $ |\nabla_g u_{\epsilon}|^2dv_g\rightharpoonup \delta _{p} $ in sense of measure, where $\delta _{p}$ is the usual Dirac measure centered at $p$.
\end{lemma}

\begin{proof}
Since $u_{\epsilon}$ is bounded in $ W^{1,2}( \Sigma , g)$,
there exists some function $u_0$ such that $u_{\epsilon}\rightharpoonup u_0$ weakly in
$W^{1,2}( \Sigma , g)$ and $u _ { \epsilon } \rightarrow u _ { 0 }$ strongly in $L^2( \partial \Sigma,\, g )$ as $\epsilon\rightarrow0$.
Obviously, $\int _ { \partial \Sigma } u _ { 0 } \mathrm { d } s _ { g } = 0$ and $\|u_0\|_{1, \alpha}\leq1$.

Suppose $u_0 \not\equiv 0$, we can see that $\int_{\Sigma}{ |\nabla_g u_0|}^2 dv_g>0$ and
$$1\geq\|u_0\|_{1, \alpha}\geq\left(1-\frac{\alpha}{\lambda_1(\partial\Sigma)}\right)\int_{\Sigma}{ |\nabla_g u_0|}^2 dv_g>0.$$
 Hence $\lVert \nabla_g ( u_{\epsilon}-u_0 ) \rVert _{2}^{2}\rightarrow \zeta:=1-\|u_0\|_{1, \alpha}$ as $\epsilon \rightarrow 0$ and $0\leq \zeta<1$.
For sufficiently small $\epsilon$, one gets $\lVert \nabla_g ( u_{\epsilon}-u_0 ) \rVert _{2}^{2}\leqslant  (\zeta+1)/2<1$.
From the H\"older inequality, there holds
\begin{equation*}
\begin{aligned}
\int_{\partial \Sigma}{e^{q( \pi -\epsilon ) u_{\epsilon}^{2}}}ds_g
&\leq \int_{\partial \Sigma}{e^{q( \pi -\epsilon ) \f( 1+\frac{1}{\delta} \ri) u_{0}^{2}+q( \pi -\epsilon ) ( 1+\delta ) ( u_{\epsilon}-u_0 ) ^2}}ds_g
\\
&\leq \f( \int_{\partial \Sigma}{e^{rq( \pi -\epsilon ) \f( 1+\frac{1}{\delta} \ri) u_{0}^{2}}}ds_g \ri) ^{{1}/{r}}\f( \int_{\partial \Sigma}{e^{sq( \pi -\epsilon ) ( 1+\delta ) \frac{\zeta +1}{2} \frac{( u_{\epsilon}-u_0 ) ^2}{\lVert \nabla _g( u_{\epsilon}-u_0 ) \rVert _{2}^{2}}}}ds_g \ri) ^{{1}/{s}}
\end{aligned}
\end{equation*}
for some $q,\ r,\ s>1$ and $1/r+1/s=1$. By the inequality (\ref{1}), we get $e^{(\pi-\epsilon)u_{\epsilon}^{2}}$ is bounded in $L^q( \partial\Sigma, g ) $ for sufficiently small $\epsilon$.
Applying the elliptic estimate to (\ref{e-u}), one gets $u_{\epsilon}$ is uniformly bounded, which contradicts $c_{\epsilon}\rightarrow +\infty$. Therefore, the assumption is not established.

Suppose $|\nabla_g u_{\epsilon}|^2dv_g\rightharpoonup \mu \ne \delta_{p}$ in sense of measure.
There exists some $r_0>0$ such that $\lim_{\epsilon \rightarrow 0}\int_{\,{B_{r_0}(p)\cap\Sigma }}{| \nabla _gu_{\epsilon} |}^2dv_g=\eta <1$. We can see that $\int_{\,{{B_{r_0}(p)\cap\Sigma }}}{| \nabla _gu_{\epsilon} |}^2dv_g\leq(\eta +1)/{2}<1$ for sufficiently small $\epsilon$.
Then we choose a cut-off function $\rho$ in $C_{0}^{1}(\,{{B_{r_0}(p)\cap\Sigma }} \,)$, which is equal to $1$ in ${\overline{B_{r_0/2}(p)\cap\Sigma }}$  such that
$\int_{\,{B_{r_0}( p ) \cap \Sigma }}{|\nabla _g( \rho u_{\epsilon} ) |^2}dv_g
\leq (\eta +3)/{4}<1$ for sufficiently small $\epsilon$.
Hence we obtain
\begin{equation*}
\begin{aligned}
\int_{{B_{r_0/2}(p)}\cap\partial\Sigma }{e^{( \pi -\epsilon ) \,q\,u_{\epsilon}^{2}}}ds_g
&\leq \int_{B_{r_0}(p)\cap\partial\Sigma }{e^{( \pi -\epsilon ) \,q\,( \rho  u_{\epsilon} ) ^2}}ds_g\\
&\leq \int_{{B_{r_0}(p)}\cap\partial\Sigma }e^{( \pi -\epsilon ) \,q\,\frac{\eta +3}{4}\frac{( \rho {u}_{\epsilon} ) ^2}{\int_{\,{B_{r_0}( p ) \cap \Sigma }}{|\nabla _g( \rho u_{\epsilon} ) |^2}dv_g}}ds_g.
\end{aligned}
\end{equation*}
By the inequality (\ref{1}), $e^{(\pi-\epsilon)u_{\epsilon}^{2}}$ is bounded in $L^q( {B_{r_0/2}(p)\cap\partial\Sigma }, g) $ for some $q>1$.
Applying the elliptic estimate to (\ref{e-u}), we get $u_{\epsilon}$ is uniformly bounded in $ {B_{r_0/4}(p)\cap\partial\Sigma }$, which contradicts $c_{\epsilon} \rightarrow +\infty$. Therefore, Lemma \ref{L3} follows.
	\end{proof}

Now we analyse the asymptotic behavior of $u_\epsilon$ near the concentration point $p$.
Let
  \begin{equation}\label{r}
  r _ { \epsilon }  = \frac { \lambda _ { \epsilon } } {c _ { \epsilon } ^ { 2 }  e ^ {  (\pi- \epsilon ) c _ { \epsilon } ^ { 2 } }}.
  \end{equation}
We take an isothermal coordinate system $(U,\phi)$ near $p$ such that $\phi(p) = 0$, $\phi$ maps $U$ to $\mathbb { R } ^ { 2  }_+=\{x=(x_1,\ x_2)\in \mathbb { R } ^ 2:\ x_2>0\}$, and $\phi (U\cap \partial\Sigma)\subset \partial\mathbb { R } ^ { 2  }_+$. In such coordinates, the metric $g$ has the representation $g = e^{2f} (dx_1^2 +dx_2^2 )$ and $f$ is a smooth function with $f (0) = 0$.
Denote $\bar{u}_\epsilon=u_\epsilon\circ\phi^{-1}$, $\bar{x}_\epsilon=\phi(x_\epsilon)$ and $U _ { \epsilon } = \{ x \in \mathbb { R } ^ { 2 }: \bar{x} _ { \epsilon } + r _ { \epsilon } x \in \phi\,(U )\}$.
Define two blowing up functions in $ U _ { \epsilon }$
\begin{equation}\label{psi}
\psi _ { \epsilon } ( x ) = \frac{\bar{u}_\epsilon (\bar{ x} _ { \epsilon } + r _ { \epsilon } x )} {c _ { \epsilon }}
\end{equation}
and
\begin{equation}\label{phi}
\varphi _{\epsilon}( x ) =c_{\epsilon}\f( \bar{u}_{\epsilon}( \bar{x}_{\epsilon}+r_{\epsilon}x ) -c_{\epsilon} \ri).
\end{equation}
From (\ref{e-u}) and (\ref{r})-(\ref{phi}), a direct computation shows
	\begin{equation}\label{e-psi}
\left\{
\begin{aligned}
  & \Delta _{\mathbb{R}^2}\psi _{\epsilon}=0\,\,\,\mathrm{in}\,\,{ \mathbb { B } _ {R} ^ { + } ( 0 ) }\\
  &\frac { \partial \psi_{\epsilon } } { \partial \mathbf{v} } =
  -e^{f( \bar{x}_{\epsilon}+r_{\epsilon}x )}  \f( c_{\epsilon}^{-2}\psi _{\epsilon}e^{( \pi -\epsilon ) ( \psi _{\epsilon}+1 ) \varphi _{\epsilon}}+\alpha r_{\epsilon}\psi _{\epsilon}-\frac{r_{\epsilon}\mu _e}{c_{\epsilon}\lambda _{\epsilon}} \ri) \,\,\,\mathrm{on}\,\,{\partial \mathbb { R } ^ { 2  }_+ \cap\partial \mathbb { B } _ {R} ^ { + } ( 0 ) }
    \end{aligned}\right.
	\end{equation}
and
	\begin{equation}\label{e-phi}
\left\{
\begin{aligned}
&\Delta_{\mathbb{R}^2}\varphi _{\epsilon}=0\,\,\,\mathrm{in}\,\,{ \mathbb { B } _ {R} ^ { + } ( 0 ) }\\
&\frac { \partial  \varphi _ { \epsilon } } { \partial \mathbf{v} } =
-e^{f( \bar{x}_{\epsilon}+r_{\epsilon}x )}  \f( \psi _{\epsilon}e^{( \pi -\epsilon ) ( \psi _{\epsilon}+1 ) \varphi _{\epsilon}}+\alpha c_{\epsilon}^{2}r_{\epsilon}\psi _{\epsilon}-\frac{c_{\epsilon}r_{\epsilon}\mu _e}{\lambda _{\epsilon}} \ri) \,\,\,\mathrm{on}\,\,{\partial \mathbb { R } ^ { 2  }_+ \cap\partial \mathbb { B } _ {R} ^ { + } ( 0 ) },
    \end{aligned}\right.
	\end{equation}	
where $\Delta_ { \mathbb { R } ^ { 2 } }$ denotes the Laplace operator on $\mathbb { R } ^ { 2 }$ and $\mathbf{v}$ denotes the outward unit normal vector on $\partial{ \mathbb { R } ^ { 2 }_+ }$.
Noticing (\ref{e-psi}), (\ref{e-phi}) and using the same argument as (\cite{Yang2007}, Lemma 4.5 and Lemma 4.6), we obtain
	\begin{equation}\label{5}
    \lim_{ \epsilon\rightarrow0}\psi _{\epsilon}= 1\,\,\mathrm{in}\, \, C_{loc}^{1}(\overline{ \mathbb{R}^2_+})
	\end{equation}
and
	\begin{equation}\label{6}
    \lim_{ \epsilon\rightarrow0}\varphi _{\epsilon}= \varphi \, \, \mathrm{in}\, \, C_{loc}^{1}( \overline{ \mathbb{R}^2_+}),
	\end{equation}
where $\varphi$ satisfies
\begin{equation*}\left\{
\begin{aligned}
&\Delta_{\mathbb{R}^2} \varphi=0,\\
&\frac{ \partial \varphi _ { \epsilon } } { \partial \mathbf{v} }  =e^{2\pi  \varphi},\\
&\varphi ( 0 ) =\sup \varphi=0.\\
\end{aligned}
\right.
\end{equation*}
By a result of Li-Zhu \cite{Li-Zhu}, we have
\begin{equation}\label{7}
\varphi ( x ) = - \frac { 1 } { 2 \pi } \log \f( \pi ^ { 2 } x _ { 1 } ^ { 2 } + ( 1 + \pi x _ { 2 } ) ^ { 2 } \ri).
\end{equation}
A direct calculation gives
\begin{equation}\label{8}
\int _ { \partial \mathbb { R } _ { + } ^ { 2 } } e ^ { 2 \pi \varphi } dx_1= 1.
\end{equation}

Next we discuss the convergence behavior of $u_\epsilon$ away from $p$.
Denote $u_{\epsilon,\,\beta }=\min   \{\ \beta c_{\epsilon} , u_{\epsilon} \}  \in W^{1,2}(\Sigma, g ) $ for any real number $0<\beta<1$. Following (\cite{Yang2007}, Lemma 4.7), we get
\begin{equation}\label{9}
\lim_{\epsilon \rightarrow 0}\| \nabla_g u _ { \epsilon ,\ \beta } \| _ { 2 } ^ { 2 }=\beta.
\end{equation}

	\begin{lemma}\label{L6}
 Letting $\lambda _{\epsilon}$ be defined by (\ref{e-u}), we obtain
 \begin{flalign}
\begin{split}
&(i)\,\limsup_{ \epsilon \rightarrow 0 }\int _ { \partial\Sigma } e ^ {(\pi-\epsilon) u_\epsilon^ 2} ds_g
= \ell(\partial\Sigma)+\lim_{\epsilon \rightarrow 0}\,\frac{\lambda _{\epsilon}}{c_{\epsilon}^{2}} \\
&(ii)\,\lim_{\epsilon \rightarrow 0}\frac{\lambda _{\epsilon}}{c_{\epsilon}^{2}}
  =\lim_{R\rightarrow +\infty}  \lim_{\epsilon \rightarrow 0}\int_{\phi ^{-1}( \mathbb{B}_{Rr_{\epsilon}}( \bar{x}_{\epsilon} ) )\cap\partial\Sigma}{e^{(\pi-\epsilon)u_{\epsilon}^{2}}}\,ds_g.
\end{split}&
\end{flalign}
	\end{lemma}
	
	\begin{proof}
Recalling (\ref{e-u}) and (\ref{9}),  one gets for any real number $0<\beta<1$,
	\begin{equation*}
\begin{aligned}
\int _ { \partial\Sigma } e ^ {(\pi-\epsilon) u_\epsilon^ 2} ds_g-\ell(\partial \Sigma )
&=
\int_{\left\{ x\in \partial\Sigma :\, u_{\epsilon}\le \beta c_{\epsilon} \right\}}{\f( e^{(\pi-\epsilon)u_{\epsilon}^{2}}-1 \ri)\, ds_g}+\int_{\left\{ x\in \partial\Sigma :\, u_{\epsilon}>\beta c_{\epsilon} \right\}}{\f( e^{(\pi-\epsilon)u_{\epsilon}^{2}}-1 \ri)\, ds_g}
\\
&\le \int_{\partial\Sigma}{\f( e^{(\pi-\epsilon)u_{\epsilon ,\,\beta}^{2}}-1 \ri)\, ds_g}+\frac{1}{\beta ^2c_{\epsilon}^{2}}\int_{\left\{ x\in \partial\Sigma :\,u_{\epsilon}>\beta c_{\epsilon} \right\}}{u_{\epsilon}^{2}e^{(\pi-\epsilon)u_{\epsilon}^{2}}\,ds_g}+o_\epsilon(1)
\\
&\le \int_{\partial\Sigma}e^{(\pi-\epsilon)u_{\epsilon ,\,\beta}^{2}}(\pi-\epsilon)u_{\epsilon}^{2}\,ds_g+\frac{\lambda _{\epsilon}}{\beta ^2c_{\epsilon}^{2}}+o_\epsilon(1)
\\
&\le \f( \int_{\partial\Sigma}e^{r(\pi-\epsilon)u_{\epsilon ,\,\beta}^{2}}ds_g \ri) ^{1/r}
\f( \int_{\partial\Sigma} (\pi-\epsilon)^su_{\epsilon}^{2s}\,ds_g \ri) ^{1/s}+\frac{\lambda _{\epsilon}}{\beta ^2c_{\epsilon}^{2}}+o_\epsilon(1),
\end{aligned}
	\end{equation*}
where $r,\ s>1$ and $1/r+1/s=1$.
By (\ref{1}) and (\ref{9}), $e^{(\pi-\epsilon)u_{\epsilon ,\,\beta}^{2}}$ is bounded in $L^r( \partial\Sigma, g )$ for some $r>1$.
Letting $\epsilon \rightarrow 0$ first and then $\beta \rightarrow 1$, we obtain
	\begin{equation}\label{s111}
\limsup _ { \epsilon \rightarrow 0 } \int _ { \partial\Sigma } e ^ {(\pi-\epsilon) u_\epsilon^ 2} ds_g-\ell(\partial \Sigma )\leq \underset{\epsilon \rightarrow 0}{\limsup}\,\frac{\lambda _{\epsilon}}{c_{\epsilon}^{2}}.
	\end{equation}
According to $c_{\epsilon}={\max}_{\overline{\Sigma}}u_{\epsilon}$, (\ref{e-u}) and Lemma \ref{L3}, we have $$
e ^ {(\pi-\epsilon) u_\epsilon^ 2} ds_g -\ell(\partial \Sigma ) \ge \int_{\partial\Sigma}{\frac{u_{\epsilon}^{2}}{c_{\epsilon}^{2}}e^{(\pi-\epsilon)u_{\epsilon}^{2}}}\,ds_g-\frac{1}{c_{\epsilon}^{2}}\int_{\partial\Sigma}{u_{\epsilon}^{2}}\,ds_g,$$
that is to say
\begin{equation}\label{s112}
\limsup _ { \epsilon \rightarrow 0 }\int _ { \partial\Sigma } e ^ {(\pi-\epsilon) u_\epsilon^ 2} ds_g -\ell(\partial \Sigma ) \ge \liminf _ {\epsilon\rightarrow0}\frac{\lambda_{\epsilon}}{c_{\epsilon}^{2}}.
\end{equation}
Combining (\ref{s111}) and (\ref{s112}), one gets the equation ($i$).

Applying (\ref{e-u}) and (\ref{r})-(\ref{phi}), we have
\begin{equation*}
  \begin{aligned}
\int_{\phi ^{-1}( \mathbb{B}_{Rr_{\epsilon}}( \bar{x}_{\epsilon} ) )\cap\partial\Sigma}{e^{(\pi-\epsilon)u_{\epsilon}^{2}}}ds_g
&=\int_{\mathbb{B}_R  \cap\partial\mathbb{R}_+^2}{r_{\epsilon}e^{(\pi-\epsilon)c_{\epsilon}^{2}}
e^{(\pi-\epsilon)( \psi _{\epsilon} +1 ) \varphi _{\epsilon}}e^{f( \bar{x}_{\epsilon}+r_{\epsilon}x )}}dx_1
\\
&=\int_{\mathbb{B}_R \cap\partial\mathbb{R}_+^2}\frac{\lambda _{\epsilon}}{\beta _{\epsilon}c_{\epsilon}^{2}}{e^{(\pi-\epsilon)( \psi _{\epsilon} +1 )\, \varphi _{\epsilon}}e^{f( \bar{x}_{\epsilon}+r_{\epsilon}x )}}dx_1.
  \end{aligned}
\end{equation*}
From (\ref{5})-(\ref{8}), ($ii$) holds.
Summarizing, we have the lemma.
    \end{proof}

Next we consider the properties of $c_{\epsilon}u_{\epsilon}$.
Combining Lemma \ref{L6} and (\cite{Yang2007}, Lemma 4.9), we obtain
\begin{equation}\label{10}
 \frac {1} { \lambda _ { \epsilon } } c _ { \epsilon } u _ { \epsilon } e ^ { (\pi-\epsilon) u _ { \epsilon } ^ { 2 } } ds_g \rightharpoonup \delta_{p}.
\end{equation}
Furthermore, one has

\begin{lemma}\label{LG}
There holds
	\begin{equation*}
	\left\{ \begin{aligned}
	&c_{\epsilon}u_{\epsilon}\rightharpoonup G\, \, \mathrm{weakly\, \, in\,}\, W^{1,q}( \Sigma, g ), \,\forall 1<q<2, \\
	&c_{\epsilon}u_{\epsilon}\rightarrow G\, \, \mathrm{strongly\, \, in  \,}\, L^2( \partial\Sigma, g ),\\
	&c_{\epsilon}u_{\epsilon}\rightarrow G\ \mathrm{in\,}\, {C_{loc}^{1}( \overline{\Sigma}\backslash \left\{ p \right\} ) },
	\end{aligned} \right.
	\end{equation*}
where $G$ is a Green function satisfying
\begin{equation}\label{e-G}
\left\{
\begin{aligned}
&\Delta_g G=\delta _{p}  \,\,\,\mathrm{in}\,\,\,\overline{\Sigma},\\
&\frac{\partial G}{\partial\mathbf{n}}=\alpha G-\frac{1}{\ell(\partial \Sigma )}\,\,\,\mathrm{on}\,\,\,\partial\Sigma\backslash\left\{p\right\},\\
&\int_{\partial\Sigma}{Gds_g}=0.
\end{aligned}
\right.
\end{equation}

\end{lemma}

\begin{proof}
From (\ref{e-u}), there holds
 	\begin{equation}\label{11}
 \left\{
\begin{aligned}
&\Delta_g ( c_{\epsilon}u_{\epsilon} ) =0\,\,\,\mathrm{in}\,\,\,\Sigma,\\
&\frac{\partial (c_{\epsilon}u_{\epsilon})}{\partial\mathbf{n}}=\frac{1}{\lambda _{\epsilon}}c_{\epsilon}u_{\epsilon}e^{(\pi-\epsilon)u_{\epsilon}^{2}}+\alpha c_{\epsilon}u_{\epsilon}-c_{\epsilon}\frac{\mu _{\epsilon}}{\lambda _{\epsilon}}\,\,\,\mathrm{on}\,\,\,\partial\Sigma\\
&\int_{\partial\Sigma}{c_{\epsilon}u_{\epsilon}\,ds_g}=0.
\end{aligned}
\right.
	\end{equation}
Integrating both side of (\ref{11}), we obtain \begin{equation}\label{01}
  \lim_{\epsilon\rightarrow0}\frac{c _ { \epsilon } \mu _ { \epsilon } }{ \lambda _ { \epsilon }}=\frac{1} { \ell ( \partial \Sigma )}.
\end{equation}
From the H\"older inequality and the Sobolev embedding theorem, one gets
\begin{equation}\label{02}
  \int_{\partial\Sigma}|c_{\epsilon}u_{\epsilon}|ds_g\leq \ell(\partial\Sigma)^{1/q^{'}}\|c_{\epsilon}u_{\epsilon}\|_{L^q{(\partial\Sigma)}}\\
  \leq C \|\nabla(c_{\epsilon}u_{\epsilon})\|_{L^q{(\Sigma)}}.
\end{equation}
It follows from Li-Liu \cite{Li-Liu} that
\begin{equation}\label{s11}
\int _ { \Sigma } | \nabla_g ( c _ { \epsilon } u _ { \epsilon } ) | ^ { q } d v _ { g } \leq \sup_{\| \Phi \| _ { H^ { 1 , q ^ { '} } } = 1}  \int _ { \Sigma } \nabla_g \Phi \nabla_g ( c _ { \epsilon } u _ { \epsilon } ) d v _ { g } ,
\end{equation}
where $1 / q + 1 / q ^ { '} = 1$.
For any $1<q<2$,
the Sobolev embedding theorem implies that $\| \Phi \| _ { C ^ { 0 } ( \Sigma ) } \leq C $, where $C$ is a constant depending only on $(\Sigma,g) .$
Using (\ref{10}), (\ref{11})-(\ref{s11}) and the divergence theorem, we have
\begin{equation*}
\begin{aligned}
\lVert \nabla_g ( c_{\epsilon}u_{\epsilon} ) \rVert _{L^q( \Sigma )}^{q}
&\leq \int_{\partial \Sigma}{\Phi}\frac{1}{\lambda _{\epsilon}}c_{\epsilon}u_{\epsilon}e^{( \pi -\epsilon ) u_{\epsilon}^{2}}ds_g+\alpha\int_{\partial\Sigma}\Phi c_{\epsilon}u_{\epsilon}ds_g-c_{\epsilon}\frac{\mu _{\epsilon}}{\lambda _{\epsilon}}\int_{\partial \Sigma}{\Phi}ds_g
\\
&\leq C \|\nabla(c_{\epsilon}u_{\epsilon})\|_{L^q{(\Sigma)}} +C,
\end{aligned}
\end{equation*}
which gives
$\|\nabla(c_{\epsilon}u_{\epsilon})\|_{L^q{(\Sigma)}} \leq C$.
The Poinca\'{r}e inequality implies that $c _ { \epsilon } u _ { \epsilon }$ is bounded in $W^ { 1 , q } ( \Sigma,g )$ for any $1 < q < 2$. Hence there exists some function $G $ such that
$c _ { \epsilon } u _ { \epsilon } \rightharpoonup G $ weakly in $W^ { 1 , q } ( \Sigma,g )$ and
$ c _ { \epsilon } u _ { \epsilon } \rightarrow G $ strongly in $L ^ { 2 } ( \partial \Sigma,g )$ as $\epsilon\rightarrow0$.
Testing $( \ref{11})$ by ${\Phi} \in C ^ { \infty } ( \overline { \Sigma } ) $, we have (\ref{e-G}).

For any fixed $\delta > 0 ,$ we choose a cut-off function $\eta \in C ^ { \infty } ( \overline { \Sigma } )$ such that $\eta \equiv 0$ on $\overline {B _ { \delta } ( p )}$ and $\eta \equiv 1$ on $\overline { \Sigma } \backslash B _ { 2 \delta } ( p ) .$
By Lemma $\ref{L3}$, we have $\lim_{\epsilon\rightarrow0}\left\| \nabla_g ( \eta u _ { \epsilon } ) \right\| _ { 2 }=0$.
Hence $\mathrm { e } ^ {(\pi-\epsilon) u _ { \epsilon } ^ { 2 } }$ is bounded in $L ^ { q } ( \overline { \Sigma } \backslash B _ { 2 \delta } ( p ))$ for any $q > 1$.
It follows from $( \ref{17})$ that
$\partial( c _ { \epsilon } u _ { \epsilon } ) / \partial \mathbf{n }\in L ^ { q _ { 0 } } ( \overline{\Sigma} \backslash B _ { 2 \delta } ( p ))$ for some $q _ { 0 } > 2 .$
Applying the elliptic estimate to (\ref{11}), we get $c _ { \epsilon } u _ { \epsilon }$ is bounded in
$C ^ { 1 } (\overline{ \Sigma} \backslash B _ { 4 \delta } ( p ))$.
This completes the proof of the lemma.
	\end{proof}

Applying the elliptic estimate to (\ref{e-G}), we can decompose $G$ near $p$ as the form
\begin{equation}\label{G}
  G=-\frac{1}{\pi}\log r+A_{p}+O(r),
\end{equation}
where $r=dist(x,p)$ and $A_{p}$ is a constant depending on $\alpha , p$ and $(\Sigma,g)$.

\subsection{Upper bound estimate}
To derive an upper bound of $\sup _ { u \in \mathcal { H } } \int _ { \partial\Sigma } e ^ {\pi u^ 2} ds_g$, we use the capacity estimate, which was first used by Li \cite{Li-JPDE} in this topic and also used by Li-Liu \cite{Yang2007}.

	\begin{lemma}\label{L8}
Under the hypotheses $c_{\epsilon}\rightarrow +\infty$ and $x _ { \epsilon } \rightarrow p \in \partial\Sigma$ as $\epsilon\rightarrow0$, there holds
	\begin{equation}\label{15}
	\sup _ { u \in \mathcal { H } } \int _ { \partial\Sigma } e ^ {\pi u^ 2} ds_g \leq \ell(\partial \Sigma ) + 2\pi e ^ { \pi A _ { p} }.
    \end{equation}
	\end{lemma}

	\begin{proof}
 We take an isothermal coordinate system $(U,\phi)$ near $p$ such that $\phi(p) = 0$, $\phi$ maps $U$ to $\mathbb { R } ^ { 2  }_+$, and $\phi (U\cap \partial\Sigma)\subset \partial\mathbb { R } ^ { 2  }_+$. In such coordinates, the metric $g$ has the representation $g = e^{2f} (dx_1^2 +dx_2^2 )$ and $f$ is a smooth function with $f(0)  = 0$. Denote $\bar{u}_\epsilon=u_\epsilon\circ\phi^{-1}$.
We claim that
\begin{equation}\label{16}
\underset{\epsilon \rightarrow 0}{\lim}\ \frac{\lambda _{\epsilon}}{c_{\epsilon}^{2}}\leq 2\pi e^{\pi A_{p}}.
\end{equation}

To confirm this claim, we set
$$W_{a,b}:=\left\{ \bar{u}_\epsilon\in W^{1,2} (\mathbb{B}^+_ \delta\setminus \mathbb{B}^+ _{Rr_\epsilon} )  :\    {\sup_ { \partial \mathbb { B } ^+_ { \delta }  \backslash \partial \mathbb { R } ^ { 2 }_+ } }\bar{u}_\epsilon =a,\, \inf _ { \partial \mathbb { B } ^+_{Rr_\epsilon}\backslash \partial \mathbb { R } ^ { 2 }_+ }\bar{u}_\epsilon=b, \,\left.\frac{\partial u}{\partial \mathbf{v}}\right|_{\partial \mathbb{R}^{2}_+ \cap (\mathbb{B}_{\delta} \backslash \mathbb{B}_{R r_{\epsilon}})}=0 \right\}$$ for some small $\delta\in(0,1)$ and some fixed $R>0$.
According to (\ref{6}), (\ref{7}), (\ref{G}) and Lemma \ref{LG}, one gets
 $$
 a=\frac{1}{c_{\epsilon}}\f( \frac{1}{\pi}\log \frac{1}{\delta}+A_{p}+o_\delta(1)+o_{\epsilon}(1)\ri)$$ and
$$
b=c_{\epsilon}+\frac{1}{c_{\epsilon}}\f( -\frac{1}{2\pi}\log ( 1+{\pi}^2 R^{2} ) +o_{\epsilon}( 1 ) \ri),$$
where $o_\delta(1)\rightarrow0$, $o_{\epsilon}(1)\rightarrow0$ as $\epsilon\rightarrow 0$ .
From a direct computation, there holds
\begin{equation}\label{17}
\pi ( a-b ) ^2=\pi c_{\epsilon}^{2}+2\log \delta -2\pi A_{p}-\log ( 1+\pi^2 R^2 ) +o_{\delta}( 1 ) +o_{\epsilon}( 1 ).
\end{equation}
The direct method of variation implies that
$\inf_{u\in W_{a,b}} \int_{\mathbb{B}^+_ \delta \setminus \mathbb{B}^+ _{Rr_\epsilon}  }{|\nabla_{\mathbb{R}^2} u|^2}dx$ can be attained by some function $m(x)\in W_{a,b}$ with $\Delta_{\mathbb{R}^2} m( x ) =0$.
We can check that
\begin{equation*}
m( x ) =\frac{a\f( \log|x|-\log ( {Rr_\epsilon} ) \ri) +b\f( \log \delta -\log|x| \ri)}{\log \delta -\log ( {Rr_\epsilon})}
\end{equation*}
and
\begin{equation}\label{19}
\int_{\mathbb{B}^+_ \delta \setminus \mathbb{B}^+ _{Rr_\epsilon}  }{|\nabla_{\mathbb{R}^2} m( x ) |^2}dx=\frac{\pi ( a-b) ^2}{\log \delta -\log ( {Rr_\epsilon} )}.
\end{equation}
Recalling (\ref{e-u}) and (\ref{r}), we have
\begin{equation}\label{20}
\log \delta -\log ( {Rr_\epsilon} ) =\log \delta -\log R-\log \frac{\lambda _{\epsilon}}{c_{\epsilon}^{2}}+{(\pi-\epsilon) c_{\epsilon}^{2} }.
\end{equation}

Letting $u^*_{\epsilon}=\max \left\{ a,\ \min \left\{b,\ \bar{u}_{\epsilon} \right\} \right\} \in W_{a,b}$, one gets $| \nabla_{\mathbb{R}^2} u^*_{\epsilon} |\leq | \nabla_{\mathbb{R}^2} \bar{u}_{\epsilon} |$ in $\mathbb{B}^+_ \delta \setminus \mathbb{B}^+ _{Rr_\epsilon}  $ for sufficiently small $\epsilon$.
These and $\|u_\epsilon\|_{1, \alpha}= 1$ lead to
\begin{equation}\label{21}
  \begin{aligned}
 \int_{\mathbb{B}^+_ \delta \setminus \mathbb{B}^+ _{Rr_\epsilon}  }{|\nabla _{\mathbb{R}^2}m( x ) |^2}dx
\leq 1+\alpha \int _{\partial\Sigma}  {u_\epsilon^2}ds_g -\int_{\Sigma \backslash \phi ^{-1}(\mathbb{B}^+ _ \delta )}{| \nabla_g u_{\epsilon} |}^2dv_g-\int_{\phi ^{-1}( \mathbb{B}^+ _{Rr_\epsilon}  )}{| \nabla_g u_{\epsilon} |}^2dv_g.
  \end{aligned}
\end{equation}
Now we compute $\int_{\Sigma \backslash \phi ^{-1}(\mathbb{B}^+ _ \delta )}{| \nabla_g u_{\epsilon} |}^2dv_g$ and $\int_{\phi ^{-1}( \mathbb{B}^+ _{Rr_\epsilon}  )}{| \nabla_g u_{\epsilon} |}^2dv_g$.
In view of (\ref{G}), we obtain
\begin{equation}\label{G3}
\int_{\Sigma \backslash \phi ^{-1}(\mathbb{B}^+ _ \delta )}{|}\nabla_g G|^2dv_g=-\frac{1}{\pi}\log \delta +A_{p}+\alpha \lVert G \rVert _{L^2(\partial\Sigma)}^{2}+o_{\epsilon}( 1 ) +o_{\delta}( 1 ).
\end{equation}
According to (\ref{phi}), (\ref{6}) and (\ref{7}), one gets
\begin{equation}\label{22}
\int_{\phi ^{-1}( \mathbb{B}^+ _{Rr_\epsilon}  )}{| \nabla_g u_{\epsilon} |}^2dv_g=\frac{1}{c_{\epsilon}^{2}}\f( \frac{1}{\pi}\log  R+\frac{1}{\pi}\log \frac{\pi}{2}+o_\epsilon( 1 )+o_{R}( 1 )\ri),
\end{equation}
where $o_R( 1 ) \rightarrow 0$ as $R\rightarrow +\infty $.
In view of (\ref{17})-(\ref{22}), we obtain
\begin{equation*}
\log \frac{\lambda _{\epsilon}}{c_{\epsilon}^{2}}\leq \log {(2\pi)}+\pi A_{p} +o(1),
\end{equation*}
where $o( 1 ) \rightarrow 0$ as $\epsilon \rightarrow 0$ first, then $R\rightarrow +\infty $ and $\delta \rightarrow 0$.
Hence the claim (\ref{16}) is confirmed.
Combining (\ref{3}), (\ref{16}) and Lemma \ref{L6}, we finish the proof of the lemma.
	\end{proof}

\subsection{Existence result}\label{sub2.4}

The content in this section is carried out under the hypothesis $0 \leq \alpha < \lambda_1(\partial\Sigma)$.
We take an isothermal coordinate system $(U,\phi)$ near $p$ such that $\phi(p) = 0$, $\phi$ maps $U$ to $\mathbb { R } ^ { 2  }_+$, and $\phi (U\cap \partial\Sigma)\subset \partial\mathbb { R } ^ { 2  }_+$. In such coordinates, the metric $g$ has the representation $g = e^{2f} (dx_1^2 +dx_2^2 )$ and $f$ is a smooth function with $f(0)  = 0$.
Set a cut-off function
$\xi \in C _ { 0 } ^ { \infty } ( B _ { 2 R \epsilon } ( p ) )$ with $\xi = 1$ on $B _ { R \epsilon } ( p)$ and $\| \nabla_g \xi \| _ { L ^ { \infty } } = O \f( 1/ (  R \epsilon )\right)$. Denote $B^+_r=\phi^{-1}(\mathbb{B}^+_r)$ and $\beta =G+1/\pi \log r-A_{p}$, where $G$ is defined as in (\ref{G}). Let $R= -\log \epsilon $, then $R\rightarrow+\infty$ and $R\epsilon\rightarrow0$ as $\epsilon\rightarrow0$.
We construct a blow-up sequence of functions
	\begin{equation}\label{23}
    v_{\epsilon}=\left\{\begin{aligned}
	&\f({ c- \frac { 1 } { 2 \pi c } \log  \frac { \pi ^ { 2 } x_1 ^ { 2 }+ ({\pi x_2}+\epsilon)^2} { \epsilon ^ { 2 } }   + \frac{ B }{c}} \right)\circ\phi\ ,&		\,\,& x\in B^+_{ R\epsilon},\\
	&\frac{G-\xi \beta}{c},&		\,\,\ \ &x\in B^+_{2R\epsilon}\backslash B^+_{R\epsilon},\\
	&\frac{G}{c},&		\,\,&x\in \Sigma \backslash B^+_{2R\epsilon},\\
\end{aligned} \right.
	\end{equation}
for some constants $B$, $c$ to be determined later,
such that
\begin{equation}\label{v}
\int _ { \Sigma } | \nabla _ { g } v_ { \epsilon } | ^ { 2 } d v _ { g } - \alpha \int _ { \partial \Sigma } (v _ { \epsilon } - \overline { v } _ { \epsilon } ) ^ { 2 } d s _ { g } = 1
\end{equation}
and $v _ { \epsilon } - \overline { v } _ { \epsilon }\in \mathcal{H}$,
where $\overline { v } _ { \epsilon } =   \int _ { \partial \Sigma } v _ { \epsilon } d s _ { g }/{ \ell ( \partial \Sigma ) }.$
Note that $\int _ { \partial \Sigma } G d s _ { g } = 0 ,$ one has $\overline { v } _ { \epsilon }= O ( R \epsilon \log ( R \epsilon ) ) ,$ and then
\begin{equation}\label{v1}
  \int _ { \partial \Sigma } \left|v _ { \epsilon } - \overline { v } _ { \epsilon } \right| ^ { 2 } d s _ { g } = \frac { \| G \| _ { L ^ { 2 } ( \partial \Sigma ) } ^ { 2 } } { c ^ { 2 } } + O \f( R \epsilon \log ^ { 2 } ( R \epsilon ) \ri).
\end{equation}

In order to assure that $v _{\epsilon}\in W^{1,2}( \Sigma, g)$, we obtain
$$  { c^2- \frac { 1 } { 2 \pi } \log ( \pi ^ { 2 } R ^ { 2 } ) + B +O\left(\frac{1}{R}\right)} ={-\frac{1}{2\pi}\log ( R\epsilon) +A_{p}},$$
which is equivalent to
	\begin{equation}\label{c1}
c^2=\frac{1}{\pi}\log \pi -\frac{1}{\pi}\log \epsilon-B +A_{p}+O\left(\frac{1}{R}\right).
	\end{equation}
A delicate calculation shows
\begin{equation*}
\begin{aligned} \int _ { B _ { R \epsilon } ^ { + } (p)} | \nabla _ { g } v _ { \epsilon } | ^ { 2 } d v _ { g }
& = \frac { 1 } { 4 \pi ^ { 2 } c ^ { 2 } } \int _ { Q ( R ) } \f| \nabla _ { \mathbb { R } ^ { 2 } } \log \f( \pi ^ { 2 } x_1 ^ { 2 } + \pi ^ { 2 } x_2 ^ { 2 } \right)\right| ^ { 2 } d x_1 d x_2 \\
& = \frac { 1 } { \pi ^ { 2 } c ^ { 2 } } \f(\pi \log(\pi R)+ \int _ {0 }^{\pi}\log \sin \theta \ d\theta-2 \int_0^{\arcsin{\frac{1}{\pi R}}}\log \sin \theta \ d\theta +O\f(\frac{\log R}{R}\ri) \ri)\\
& = \frac { 1 } { \pi c ^ { 2 } } \f( \log R + \log \frac { \pi } { 2 } + O \f( \frac { \log R } { R } \ri) \ri), \end{aligned}
\end{equation*}
where $Q ( R ) =  \{ ( x_1 , x_2 ): ( x_1 , x_2-1/\pi ) \in \mathbb { B } _ { R } ^+ ,\ x_2 \geq 1 / \pi  \}$.
According to (\ref{23}) and (\ref{G3}), one has
\begin{equation*}
  \int _ { \Sigma \backslash  B_ {R\epsilon }^+(p) } | \nabla _ { g } v _ { \epsilon } | ^ { 2 } d v _ { g } = \frac {1}{ c ^ { 2 } }\f( A _ { p } + \alpha \| G \| _ { L ^ { 2 } ( \partial \Sigma ) } ^ { 2 } - \frac { 1 } { \pi } \log ( R \epsilon ) + O \f( R \epsilon \log ^ { 2 } ( R \epsilon ) \ri) \ri).
\end{equation*}
Then we get
\begin{equation}\label{v2}
\int _ { \Sigma  } | \nabla _ { g } v _ { \epsilon } | ^ { 2 } d v _ { g } = \frac {1}{ c ^ { 2 } }\f( A _ { p } + \alpha \| G \| _ { L ^ { 2 } ( \partial \Sigma ) } ^ { 2 }  +\frac { 1 } { \pi } \log \f(\frac{\pi}{2\epsilon} \ri)+O \f( \frac { \log R } { R } \ri)+ O \f( R \epsilon \log ^ { 2 } ( R \epsilon ) \ri) \ri).
\end{equation}
In view of (\ref{v}), (\ref{v1}), (\ref{v2}), there holds
\begin{equation}\label{c2}
{c}^2=A_{p}+\frac { 1 } { \pi } \log \f(\frac{\pi}{2\epsilon} \ri)+O\f( \frac { \log R } { R } \ri) +O\f( R\epsilon \log^2 ( R\epsilon ) \ri).
\end{equation}
According to (\ref{c1}) and (\ref{c2}), one gets
	\begin{equation*}\label{B}
B = \frac { 1 } {  \pi }\log 2 + O\f( \frac { \log R } { R } \ri) +O\f( R\epsilon \log^2 ( R\epsilon )\ri).
	\end{equation*}
It follows that in $\partial\Sigma \cap \partial B _ { R \epsilon } ^+(p)$,
$$
 \pi ( v _ { \epsilon } - \overline { v } _ { \epsilon } ) ^ { 2 }  \geq   \log ( 2 \pi ) + \pi A _ { p } - \log \f(  \frac { \epsilon^2+\pi ^ { 2 } x_1 ^ 2 } { \epsilon  } \ri)  + O\f( \frac { \log R } { R } \ri) + O \f( R \epsilon \log ^ {2} (R \epsilon)\ri).
$$
Hence
	\begin{equation}\label{27}
\int_{\partial\Sigma \cap \partial B _ { R \epsilon } ^+(p)}{e}^{\pi ( v _ { \epsilon } - \overline { v } _ { \epsilon } ) ^ { 2 }}\,ds_g
\ge 2\pi e^{\pi A_{p}}+O \f( \frac { \log R } { R } \ri) + O \f( R \epsilon \log ^ {2} ( R \epsilon ) \ri).
	\end{equation}	
On the other hand, from the fact $e^t\geq t+1$ for any $t>0$ and (\ref{23}), we get
	\begin{equation}\label{24}
	\begin{aligned}
\int_{\partial\Sigma \backslash \partial B_{R\epsilon}^+( p)}{e}^{\pi ( v _ { \epsilon } - \overline { v } _ { \epsilon } ) ^ { 2 }}\,ds_g
&\ge \int_{\partial\Sigma \backslash \partial B_{R\epsilon}^+( p)}{\f( 1+\pi ( v _ { \epsilon } - \overline { v } _ { \epsilon } ) ^2\ri)}\,ds_g\\
    &\ge \ell(\partial \Sigma ) +\frac{\pi \| G \| _ { L ^ { 2 } ( \partial \Sigma ) } ^ { 2 }}{c^2}+O\f( R \epsilon \log ^ {2} ( R \epsilon ) \ri).
	\end{aligned}
	\end{equation}
From (\ref{27}) and (\ref{24}), there holds
	\begin{equation}\label{25}
    \begin{aligned}
\int_{\partial\Sigma }{e}^{\pi ( v _ { \epsilon } - \overline { v } _ { \epsilon } ) ^ { 2 }}\,ds_g>\ell(\partial \Sigma ) +2\pi e^{\pi A_{p}}
    \end{aligned}	
    \end{equation}
for sufficiently small $\epsilon$.
The contradiction between (\ref{15}) and (\ref{25}) indicates that $c_{\epsilon}$ must be bounded, which together with elliptic estimates completes the proof of Theorem \ref{T1}.

$\hfill\Box$

\section{Higher order eigenvalue cases}
In this section, we will prove Theorem \ref{2T1} involving higher order eigenvalues through blow-up analysis.
Let $\ell$ be a positive integer and $E_\ell(\partial\Sigma)$ be defined by (\ref{l1}).
Denote the dimension of $E_\ell(\partial\Sigma)$ is $s_{\ell}$.
From (\cite{Brezis}, Theorem 9.31), it is known that $s_{\ell}$ is a finite constant depending only on $\ell$. Then we can find a set of normal orthogonal basis $\{e_i\in C^{\infty}(\overline{\Sigma}),\ 1\leq i\leq s_{\ell}\}$ of $E_\ell(\partial\Sigma)$ satisfying
\begin{equation}\label{ee}
\left\{
\begin{aligned}
&\int _ { \partial\Sigma }  e _ { i} \, d s_g =0,\\
& \Delta_g e_i =0 \ \ \mathrm{in}\ \  \Sigma.\\
\end{aligned}
\right.
\end{equation}

	\subsection{Blow-up analysis}

  Let $\lambda_{\ell+1}(\partial\Sigma)$ and $\mathcal { S }$ be defined by (\ref{lal}) and (\ref{S}). In view of Lemma \ref{L1} and (\ref{ee}), we have
\begin{lemma}\label{2L1}
Let $0 \leq \alpha < \lambda_{\ell+1}(\partial\Sigma)$ be fixed. For any $0<\epsilon <\pi$, the supremum
$\sup_{ u \in \mathcal { S} }\int _ { \partial\Sigma } e ^ {(\pi-\epsilon) u^ 2} ds_g$ is attained by
some function $u_\epsilon \in \mathcal{S}  \cap C^{\infty} (\overline{ \Sigma})$.
Moreover, the Euler-Lagrange equation of $u_{\epsilon}$ is
\begin{equation}\label{2e-u}
	\left\{
\begin{aligned}
&\Delta_g u _ { \epsilon } = 0 \,\,\, \mathrm{in}\,\,\, \Sigma ,\\
& \frac { \partial u _ { \epsilon } } { \partial \mathbf{n} } = \frac { 1 } { \lambda _ { \epsilon } } u _ { \epsilon } e ^ {(\pi- \epsilon) u _ { \epsilon } ^ { 2 } } + \alpha u _ { \epsilon } - \frac { \mu _ { \epsilon } } { \lambda _ { \epsilon } }-\sum^{s_{\ell}}_{i=1}\frac {\beta _ { \epsilon , i}}{\lambda _ { \epsilon }}e_i \,\,\, \mathrm{ on } \,\,\, \partial \Sigma ,\\
&\lambda _ { \epsilon }= {\int _ { \partial \Sigma } u _ { \epsilon } ^ { 2 } \mathrm { e } ^ { (\pi- { \epsilon }) u _ { \epsilon } ^ { 2 } } ds _ { g }} ,\\
&\mu _ { \epsilon }= \frac {1} { \ell ( \partial \Sigma ) } \int _ { \partial \Sigma } u _ { \epsilon } \mathrm { e } ^ { (\pi-\epsilon) u _ { \epsilon } ^ { 2 } }ds _ { g },\\
&\beta _ { \epsilon ,i}= \int _ { \partial \Sigma } u _ { \epsilon } \mathrm { e } ^ { (\pi-\epsilon) u _ { \epsilon } ^ { 2 } }e_i ds _ { g }.
 \end{aligned} \right.
	\end{equation}
\end{lemma}

We now perform the blow-up analysis. Let $c_{\epsilon}=|u_{\epsilon}( x_\epsilon )| =\max_{\overline{\Sigma}}|u_{\epsilon}|$.
With no loss of generality, we assume in the following
$c_{\epsilon}=u_{\epsilon}( x_\epsilon ) \rightarrow + \infty$ and $x_\epsilon\rightarrow p$ as $\epsilon\rightarrow0$.
Applying maximum principle to (\ref{2e-u}), we have $p \in \partial \Sigma$.
Analogous to Lemma 4, we get

\begin{lemma}\label{2LG}
There holds
$c_{\epsilon}u_{\epsilon}\rightharpoonup G$ weakly in $W^{1,q}( \Sigma, g )\ (\forall 1<q<2)$,
$c_{\epsilon}u_{\epsilon}\rightarrow G$ strongly in $L^2( \partial\Sigma, g )$ and
$c_{\epsilon}u_{\epsilon}\rightarrow G$ in $ {C_{loc}^{1}( \overline{\Sigma}\backslash \left\{ p \right\} ) }$ as $\epsilon\rightarrow0$,
where $G$ is a Green function satisfying
\begin{equation*}\label{2e-G}
\left\{
\begin{aligned}
&\Delta_g G=\delta _{p}  \,\,\,\mathrm{in}\,\,\,\overline{\Sigma},\\
&\frac{\partial G}{\partial\mathbf{n}}=\alpha G-\frac{1}{\ell(\partial \Sigma )}-\sum^{s_{\ell}}_{i=1} e_i e_i(p) \,\,\,\mathrm{on}\,\,\,\partial\Sigma\backslash\left\{p\right\},\\
&\int_{\partial\Sigma}{Gds_g}=0.
\end{aligned}
\right.
\end{equation*}
\end{lemma}

Moreover, $G$ near $p$ can be decomposed into
\begin{equation}\label{2G}
  G=-\frac{1}{\pi}\log r+A_{p}+O(r),
\end{equation}
where $r=dist(x,p)$ and $A_{p}$ is a constant depending on $\alpha , p$ and $(\Sigma,g)$.
Analogous to Lemma \ref{L8}, using the capacity estimate, we derive an upper bound of the supremum (\ref{sup2}):
	\begin{equation}\label{215}
	\sup _ { u \in \mathcal { S} } \int _ { \partial\Sigma } e ^ {\pi u^ 2} ds_g \leq \ell(\partial \Sigma) + 2\pi e ^ { \pi A _ { p} }.
    \end{equation}

\subsection{Existence result}
The content is carried out under the hypothesis $0 \leq \alpha < \lambda_{\ell+1}(\partial\Sigma)$.
We take an isothermal coordinate system $(U,\phi)$ near $p$ such that $\phi(p) = 0$, $\phi$ maps $U$ to $\mathbb { R } ^ { 2  }_+$, and $\phi (U\cap \partial\Sigma)\subset \partial\mathbb { R } ^ { 2  }_+$. In such coordinates, the metric $g$ has the representation $g = e^{2f} (dx_1^2 +dx_2^2 )$ and $f$ is a smooth function with $f (0) = 0$.
Set a cut-off function
$\xi \in C _ { 0 } ^ { \infty } ( B _ { 2 R \epsilon } ( p ) )$ with $\xi = 1$ on $\overline{B _ { R \epsilon } ( p)}$ and $\| \nabla_g \xi \| _ { L ^ { \infty } } = O (  1/(  R \epsilon ) )$. Denote $B^+_r=\phi^{-1}(\mathbb{B}^+_r)$ and $\beta =G+(1/\pi)\log r-A_{p}$, where $G$ is defined as in (\ref{2G}). Let $R= -\log \epsilon $, then $R\rightarrow+\infty$ and $R\epsilon\rightarrow0$ as $\epsilon\rightarrow0$.
We construct a blow-up sequence of functions
	\begin{equation*}\label{223}
    v_{\epsilon}=\left\{\begin{aligned}
	&\f({ c- \frac { 1 } { 2 \pi c} \log  \frac { \pi ^ { 2 } x_1 ^ { 2 }+ ({\pi x_2}+\epsilon)^2} { \epsilon ^ { 2 } }   + \frac{B}{c} }\ri)\circ\phi \ ,&		\,\,& x\in B^+_{ R\epsilon},\\
	&\frac{G-\xi \beta}{c},&		\,\,\ \ &x\in B^+_{2R\epsilon}\backslash B^+_{R\epsilon},\\
	&\frac{G}{c},&		\,\,&x\in \Sigma \backslash B^+_{2R\epsilon},\\
\end{aligned} \right.
	\end{equation*}
for some constants $B$, $c$ to be determined later,
such that
$$
\int _ { \Sigma } \left| \nabla _ { g } v_ { \epsilon } \right| ^ { 2 } d v _ { g } - \alpha \int _ { \partial \Sigma } (v _ { \epsilon } - \overline { v } _ { \epsilon } ) ^ { 2 } d s _ { g } = 1$$
and $v _ { \epsilon } - \overline { v } _ { \epsilon }\in \mathcal{S}$,
where $\overline { v } _ { \epsilon } =   \int _ { \partial \Sigma } v _ { \epsilon } d s _ { g }/{ \ell ( \partial \Sigma ) }.$
Similar to the subsection \ref{sub2.4},
we determine the constants
	\begin{equation*}\label{2B}
B = \frac { 1 } {  \pi }\log 2 + O\f( \frac { \log R } { R } \ri) +O\f( R\epsilon \log^2 ( R\epsilon )\ri)
	\end{equation*}
and
\begin{equation*}\label{2c2}
{c}^2=A_{p}+\frac { 1 } { \pi } \log \f(\frac{\pi}{2\epsilon} \ri)+O\f( \frac { \log R } { R } \ri) +O\f( R\epsilon \log^2 ( R\epsilon ) \ri).
\end{equation*}
Then we get
	\begin{equation}\label{227}
\int_{\partial\Sigma }{e}^{\pi ( v _ { \epsilon } - \overline { v } _ { \epsilon } ) ^ { 2 }}\,ds_g
\ge 2\pi e^{\pi A_{p}}+ \ell( \partial\Sigma ) +\frac{\pi \| G \| _ { L ^ { 2 } ( \partial \Sigma ) } ^ { 2 }}{c^2}+O \f( R \epsilon \log ^ {2} ( R \epsilon ) \ri)+O \f( \frac { \log R } { R } \ri).
	\end{equation}

Following Yang \cite{Yang-JDE-15}, we set $${v}^*_\epsilon=(v _ { \epsilon } - \overline { v } _ { \epsilon })-
\sum^{s_{\ell}}_{i=1}e_{i}\int_{\partial\Sigma}(v _ { \epsilon } - \overline { v } _ { \epsilon }) e_i\ ds_g\in E_{\ell}^{\bot},$$
which gives
\begin{equation*}
  \left\{\begin{aligned}
  	&{v}^*_\epsilon=(v _ { \epsilon } - \overline { v } _ { \epsilon })+O\f(\frac{1}{R^2}\ri),\\
  &\|{v}^*_\epsilon\|^2_{1,\alpha}=1+O\left(\frac{1}{R^2}\right),\\
	&\int _{\partial\Sigma}  {v}^*_\epsilon\,ds_g = 0.\\
\end{aligned} \right.
	\end{equation*}
It is easy to verify $\tilde{v}_\epsilon={v}^*_\epsilon/\|{v}^*_\epsilon\|^2_{1,\alpha}\in \mathcal{S}$.
According to this and (\ref{227}), we have
	\begin{equation}\label{225}
    \begin{aligned}
\int_{\partial\Sigma }{e}^{\pi  \tilde{v} _ { \epsilon }^2 }\,ds_g&\geq
\f(1+O\f(\frac{1}{R^2}\ri)\ri)\int_{\partial\Sigma }{e}^{\pi ( v _ { \epsilon } - \overline { v } _ { \epsilon } ) ^ { 2 }}\,ds_g\\
&\geq 2\pi e^{\pi A_{p}}+ \ell( \partial \Sigma ) +\frac{\pi \| G \| _ { L ^ { 2 } ( \partial \Sigma ) } ^ { 2 }}{c^2}+O \f( R \epsilon \log ^ {2} ( R \epsilon ) \ri)+O \f( \frac { \log R } { R } \ri)\\
&>2\pi e^{\pi A_{p}}+\ell(\partial \Sigma )
    \end{aligned}	
    \end{equation}
for sufficiently small $\epsilon$.
The contradiction between (\ref{215}) and (\ref{225}) indicates that $c_{\epsilon}$ must be bounded, which together with elliptic estimates completes the proof of Theorem \ref{2T1}.

$\hfill\Box$

\end{document}